\documentclass[11pt]{amsart}

\newtheorem{theorem}{Theorem}[section]

\newtheorem{lemma}[theorem]{Lemma}

\usepackage[colorlinks, bookmarks=true]{hyperref}
\usepackage{color,graphicx,shortvrb}
\usepackage[active]{srcltx} 

\def\11{\textbf{$1$}}

\usepackage{enumerate}
\usepackage{amsmath}
\usepackage{amssymb}
\usepackage[latin 1]{inputenc}

\begin{document}

\title{Bilocal $^*$-automorphisms of $B(H)$ satisfying the 3-local property}

\author[Ben Ali Essaleh]{Ahlem Ben Ali Essaleh}
\email{ahlem.benalisaleh@gmail.com}
\address{Faculte des Sciences de Monastir, Département de Mathématiques, Avenue de L'environnement, 5019 Monastir, Tunisia}
\curraddr{Departamento de An{\'a}lisis Matem{\'a}tico, Facultad de
Ciencias, Universidad de Granada, 18071 Granada, Spain.}

\author[Niazi]{Mohsen Niazi}
\email{mhsniazi@yahoo.com}
\address{Faculty of Mathematical Sciences in University of Birjand, Birjand, Iran}
\curraddr{Departamento de An{\'a}lisis Matem{\'a}tico, Facultad de
Ciencias, Universidad de Granada, 18071 Granada, Spain.}

\author[Peralta]{Antonio M. Peralta}
\address{Departamento de An{\'a}lisis Matem{\'a}tico, Facultad de
Ciencias, Universidad de Granada, 18071 Granada, Spain.}
\curraddr{Visiting Professor at Department of Mathematics, College of Science, King Saud University, P.O.Box 2455-5, Riyadh-11451, Kingdom of Saudi Arabia.}
\email{aperalta@ugr.es}

\thanks{Authors partially supported by the Spanish Ministry of Science and Innovation, D.G.I. project no. MTM2011-23843, and Junta de Andaluc\'{\i}a grant FQM375. Third author partially supported by the Deanship of Scientific Research at King Saud University (Saudi Arabia) research group no. RG-1435-020. The first and second authors acknowledge the partial financial support from the IEMath-GR program for visits of young talented researchers. First author partially supported by the Higher Education And Scientific Research In Tunisia project UR11ES52.}



\begin{abstract} We prove that, for a complex Hilbert space $H$ with dimension bigger or equal than three,  every linear mapping $T: B(H)\to B(H)$ satisfying the 3-local property is a $^*$-monomorphism, that is, every linear mapping $T: B(H) \to B(H)$ satisfying that for every $a$ in $B(H)$ and every $\xi,\eta$ in $H$, there exists a $^*$-automorphism $\pi_{a,\xi,\eta}: B(H)\to B(H)$, depending on $a$, $\xi$, and $\eta$, such that $$T(a) (\xi) = \pi_{a,\xi,\eta} (a) (\xi), \hbox{ and } T(a) (\eta) = \pi_{a,\xi,\eta} (a) (\eta),$$ is a $^*$-monomorphism. This solves a question posed by L. Molnár in [\emph{Arch. Math.} \textbf{102}, 83-89 (2014)].
\end{abstract}


\maketitle
 \thispagestyle{empty}

\section{Introduction and preliminaries}

In a recent contribution (cf. \cite{Mol2014}), L. Moln\'{a}r devotes a note to release and develop the notion of \emph{bilocal $^*$-automorphisms} on $B(H)$,  a notion which was inspired by the concept of bilocal derivations due to Xiong and Zhu (see \cite{ZhuXiong97}). According to \cite{Mol2014}, a linear mapping $T: B(H) \to B(H)$ is said to be a \emph{bilocal $^*$-automorphism} if for every $a$ in $B(H)$ and every $\xi$ in $H$, there exists a $^*$-automorphism $\pi_{a,\xi}: B(H)\to B(H)$, depending on $a$ and $\xi$, such that $T(a) (\xi) = \pi_{a,\xi} (a) (\xi)$.\smallskip

We recall that a linear mapping $T$ on a C$^*$-algebra $A$ is said to be a \emph{local $^*$-automorphism} if for every $a$ in $A$, there exists a $^*$-automorphism $\pi_{a}: A\to A$, depending on $a$, such that $T(a)= \pi_{a,\xi} (a)$ (compare \cite{LarSou} and \cite{BreSemrl95}). Clearly, every local  $^*$-automorphism on $B(H)$ is a bilocal  $^*$-automorphism. \smallskip

The most significant results established by Moln\'{a}r in \cite{Mol2014} can be subsumed in the following:

\begin{theorem}\label{t Molnar 2014} \cite[Theorem 1]{Mol2014} Let $H$ be an infinite dimensional and separable complex Hilbert space. For a linear transformation $T: B(H)\to B(H)$ the following two assertions are equivalent:\begin{enumerate}[$(a)$]\item $T$ is a bilocal $^*$-automorphism, that is, for every $a\in B(H)$ and $\xi \in H$, there exists an algebra $^*$-automorphism $\pi_{a,\xi}$ of $B(H)$ such that $T(a) (\xi) = \pi_{a,\xi} (a) (\xi)$;
\item $T$ is a unital algebra $^*$-endomorphism of $B(H)$.
\end{enumerate} Moreover, if $H$ is finite-dimensional, then $(a)$ holds if and only if $T$ is either an algebra
$^*$-automorphism or an algebra $^*$-anti-automorphism of $B(H)$.$\hfill\Box$
\end{theorem}

We know that the linear bijection $T: M_2 (\mathbb{C}) \to M_2 (\mathbb{C})$, $T(a) = a^{t}$, where $a^{t}$ denotes the transpose of $a$, is a local $^*$-automorphism, and a Jordan $^*$-automorphism, but it is not multiplicative (compare \cite[Example 3.14]{Pe2014}). So, the conclusion of the above Theorem cannot be improved easily.\smallskip

Bilocal $^*$-automorphisms can be englobed in a wider class of linear maps which are called \emph{strong-local $^*$-automorphisms} (cf. \cite{BenAliPeraltaRamirez}). Given a C$^*$-algebra $A$, we denote by $S(A)$ the set of all states on $A$ (i.e. the set of all norm-one, positive functionals in $A^*$). For each $\phi\in S(A)$, the assignment $a\mapsto \||a|\|_{\phi} = \phi (a^* a)^{\frac12},$ defines a pre-Hilbertian seminorm on $A$. A linear mapping $T: A \to A$ is called a \emph{strong-local $^*$-automorphism} if for every $a\in A$, and every state $\phi\in S(A)$, there exists a $^*$-automorphism $\pi_{a,\phi}: A\to A,$ depending on $a$ and $\phi$, such that $$\||T(a) - \pi_{a,\phi} (a)|\|_{\phi }=0.$$ With a little abuse of notation, when in the above definition, $A$ is a von Neumann algebra $M$ and the set $S(A)$ of states of $A$ is replaced with the set $S_n(M)$ of all normal states on $M$, we also employ the term \emph{strong-local $^*$-automorphism} on $M$ (cf. \cite{BenAliPeraltaRamirez}).\smallskip

There exists another interesting subclass of the class of strong-local $^*$-automorphisms defined as follows: Let $M$ be von Neumann algebra. A linear mapping $T: M \to M$ is said to be an \emph{extreme-strong-local $^*$-automorphism} if for every $a\in M$, and every pure normal state $\phi\in \partial_{e} (S_n(M))$, there exists a $^*$-automorphism $\pi_{a,\phi}: M\to M,$ depending on $a$ and $\phi$, such that $$\||T(a) - \pi_{a,\phi} (a)|\|_{\phi }=0,$$ (compare \cite{BenAliPeraltaRamirez}).\smallskip

It is shown in \cite{BenAliPeraltaRamirez} that bilocal $^*$-automorphisms and extreme-strong-local $^*$-automorphisms on $B(H)$ define the same applications.\smallskip

Theorem 4.3 in \cite{BenAliPeraltaRamirez} established that every extreme-strong-local $^*$-auto-morphism\hyphenation{auto-morphism} on an atomic von Neumann algebra is a Jordan $^*$-homomorphism. The conclusion of this result is no longer true for general von Neumann algebras, even nor for von Neumann algebras which are bidual Banach spaces (compare \cite[comments preceding Theorem 4.3]{BenAliPeraltaRamirez}).\smallskip

In an attempt to find additional hypothesis to determine when a bilocal $^*$-automorphism is a $^*$-homomorphism (i.e. multiplicative),  Moln\'{a}r introduced in \cite[Corollary 2]{Mol2014} the following ``3-local property'': a linear mapping $T: B(H)\to B(H)$ satisfies the \emph{3-local property} if for every $a$ in $B(H)$ and every $\xi,\eta$ in $H$, there exists a $^*$-automorphism $\pi_{a,\xi,\eta}: B(H)\to B(H)$, depending on $a$, $\xi$, and $\eta$ such that $$T(a) (\xi) = \pi_{a,\xi,\eta} (a) (\xi), \hbox{ and } T(a) (\eta) = \pi_{a,\xi,\eta} (a) (\eta).$$ Clearly every linear mapping with the 3-local property also is a  bilocal $^*$-automorphism or a extreme-strong-local $^*$-automorphism.\smallskip

Corollary 2 in \cite{Mol2014} asserts that, for an infinite dimensional separable complex Hilbert space $H$, every linear mapping $T: B(H)\to B(H)$ satisfying the 3-local property is a $^*$-homomorphism. The problem whether, in \cite[Corollary 2]{Mol2014}, the same conclusion remains or not true when $H$ is finite dimensional or infinite dimensional and non-separable, was left as an open problem by Moln\'{a}r (cf. \cite[pages 87-88]{Mol2014}). It was already noticed in \cite{Mol2014}, and in this introduction, that the linear bijection $T: M_2 (\mathbb{C}) \to M_2 (\mathbb{C})$, $T(a) = a^{t}$ is a local $^*$-automorphism which is not multiplicative (compare \cite[Example 3.14]{Pe2014}), so the result is not true for when $H$ is $2$-dimensional.\smallskip

In this paper we provide a complete positive answer to the above question for any complex Hilbert space $H$ with dimension bigger or equal than three (see Theorem \ref{t solution Molnar}). The proof is based in an independent result which shows that for every complex Hilbert space $H$ with dim$(H)\geq 3$, we can always find an operator $a\in B(H)$ which is not bilocally unitarily equivalent to its transpose, that is there exist $\xi_0$ and $\eta_0$ in $H$ satisfying $$\| a_0^t (\eta_0) - u a_0 u^* (\eta_0)\| + \| a_0^t (\xi_0) -u a_0 u^* (\xi_0)\| \neq 0,$$ for every unitary $u\in B(H)$ (cf. \ref{t non bilocally unitarily equivalent operators}).

\section{2-local unitarily equivalent matrices and operators}

Let $B(H)$ denote the C$^*$-algebra of all bounded linear operators on an arbitrary complex Hilbert space $H$. Following standard notation, we shall say that an operator $a$ in $B(H)$ is \emph{unitarily equivalent to its transpose}, $a^{t}$, when there exists a unitary $u\in B(H)$ satisfying $a = u a^t u^*$. Originated by a question of P.R. Halmos, who asked in \cite[Proposition 159]{Hal54} whether every square complex matrix is unitarily equivalent to its transpose, a question completely solved by W. Specht \cite{Specht}, the study of matrices which are unitarily equivalent to their transposes attracted the attention of many researchers. We refer to the papers \cite{GarTener2012,IkraAbdi2012} as good recent references on results about matrices which are unitarily equivalent to their transposes.\smallskip

We are interested in a priori weaker property: the operator $a$ is \emph{bilocally unitarily equivalent to its transpose} when for every $\xi,\eta\in H$, there exists a unitary $u_{\xi,\eta}\in B(H)$, depending on $\xi$ and $\eta$, such that $$ a(\xi) = u_{\xi,\eta} a^t u_{\xi,\eta}^* (\xi),\hbox{ and } a (\eta) =u_{\xi,\eta} a^t u_{\xi,\eta}^* (\eta).$$

Our first Theorem shows the existence of operators which are not bilocally unitarily equivalent to their transposes.

\begin{theorem}\label{t non bilocally unitarily equivalent operators} Let $H$ be a complex Hilbert space with dim$(H)\geq 3$. Then there exists $a\in B(H)$ which is not bilocally unitarily equivalent to its transpose. Furthermore, suppose $K$ is a closed subspace of $H$ with dim$(K)\geq 3$, and let $p$ denote the orthogonal projection of $H$ onto $K.$ Then there exists $a\in B(H)$ with $a= pap,$ which is not bilocally unitarily equivalent to its transpose in $B(H)$.
\end{theorem}

\begin{proof} Let us take an orthonormal basis $\{\xi_1,\xi_2,\xi_3\}\cup \{\xi_j\}_{j\in J}$ of $H$. We observe that $J$ can be an empty, a finite or an infinite set. We consider the operator $a = \xi_1 \otimes \xi_2 + 2 \xi_2\otimes \xi_3\in B(H)$. We shall prove that $a$ is not bilocally unitarily equivalent to its transpose. Suppose, on the contrary, that $a$ is bilocally unitarily equivalent to its transpose.  By assumptions, for each $\eta\in H$ there exists a unitary $u_{\xi_3,\eta}\in B(H)$ satisfying $$0=u_{\xi_3,\eta} a^t(\xi_3) = a u_{\xi_3,\eta} (\xi_3),$$ which implies that  $u_{\xi_3,\eta} (\xi_3)\in \ker (a) = \{\xi_2,\xi_3\}^{\perp}.$ This shows that \begin{equation}\label{eq perp to xi2 3} \langle u_{\xi_3,\eta} (\xi_3) / \xi_i \rangle =0,
\end{equation} for every $i=2,3$, and $\eta\in H$.\smallskip

By assumptions, there is a unitary $u_{\xi_3,\xi_2}\in B(H)$ satisfying $$u_{\xi_3,\xi_2} a^t(\xi_3) = a u_{\xi_3,\xi_2} (\xi_3),$$ and $$ u_{\xi_3,\xi_2} a^t(\xi_2) = a u_{\xi_3,\xi_2} (\xi_2).$$ Since $2 u_{\xi_3,\xi_2} (\xi_3)= u_{\xi_3,\xi_2} a^t(\xi_2) = a u_{\xi_3,\xi_2} (\xi_2)\in a(H) = \mathbb{C} \xi_1\oplus\mathbb{C} \xi_2,$ we deduce that $$ 2 u_{\xi_3,\xi_2} (\xi_3)= \mu_1 \xi_1 + \mu_2 \xi_2,$$ where $|\mu_1|^2 + |\mu_2|^2 = 4$, because $u_{\xi_3,\xi_2}$ is a unitary. It follows from \eqref{eq perp to xi2 3} that $2 u_{\xi_3,\xi_2} (\xi_3)= \mu_1 \xi_1 $ with $|\mu_1|=2$.\smallskip

On the other hand $$u_{\xi_3,\xi_2} (\xi_2) = \lambda_1 \xi_1 +\lambda_2 \xi_2 +\lambda_3 \xi_3 + \sum_{j\in J} \lambda_j \xi_j, $$ where the family $(\lambda_j)_{j\in J}$ is sumable and $\displaystyle \sum_{j\in J} |\lambda_j|^2 + |\lambda_1|^2 + |\lambda_2|^2 + |\lambda_3|^2 =1,$ because $u_{\xi_3,\xi_2}$ is a unitary. Therefore $$\mu_1 \xi_1=2 u_{\xi_3,\xi_2} (\xi_3)=  a u_{\xi_3,\xi_2} (\xi_2) = \lambda_2 \xi_1 + 2 \lambda_3 \xi_2,$$ and hence $\lambda_3=0$, and $\mu_1 \xi_1= \lambda_2 \xi_1$, which is impossible because $2=|\mu_1| = |\lambda_2| \leq 1$.
\end{proof}

We can state now a result which provides a complete solution to the question posed by Moln\'{a}r in \cite[pages 87-88]{Mol2014}.

\begin{theorem}\label{t solution Molnar} Let $H$ be a complex Hilbert space with dim$(H)\geq 3$. Let $T: B(H)\to B(H)$ be a linear mapping satisfying the 3-local property, that is, for every $\xi,\eta\in H$, and every $a\in B(H)$, there exists a $^*$-automorphism $\pi_{\xi,\eta,a} : B(H) \to B(H)$ satisfying $$T(a) (\xi) =  \pi_{\xi,\eta,a} (a) (\xi), \hbox{ and, } T(a) (\eta) =  \pi_{\xi,\eta,a} (a) (\eta),$$ equivalently {\rm(}compare \cite[Theorem A.8]{Mol07}{\rm)}, for every $\xi,\eta\in H$, and every $a\in B(H)$, there exists a unitary $u_{\xi,\eta,a}$ in  $B(H)$ satisfying $$T(a) (\xi) =  u_{\xi,\eta,a} a u_{\xi,\eta,a}^* (\xi), \hbox{ and } T(a) (\eta) =  u_{\xi,\eta,a} a u_{\xi,\eta,a}^* (\eta).$$ Then $T$ is a $^*$-monomorphism.
\end{theorem}

Before proving the above theorem we state a technical result borrowed from \cite[Proof of Corollary 1]{Mol2014}, the proof is included here for completeness reasons.

\begin{lemma}\label{l tech lemma rank one} Let $T: B(H)\to B(H)$ be a linear mapping satisfying the 3-local property, where $H$ is a complex Hilbert space. Then $T$ maps rank-one operators to rank-one operators.
\end{lemma}

\begin{proof} Let $a$ be a rank-one operator in $B(H).$ Given $\xi,\eta$ in $H$, there exists a unitary $u_{\xi,\eta,a}$ in  $B(H)$ satisfying $$T(a) (\xi) =  u_{\xi,\eta,a} a u_{\xi,\eta,a}^* (\xi), \hbox{ and } T(a) (\eta) =  u_{\xi,\eta,a} a u_{\xi,\eta,a}^* (\eta).$$ This implies that $T(a) (\xi)$ and $T(a) (\eta)$ are linearly dependent for every $\xi$ and $\eta$ in $H$, which concludes the proof.
\end{proof}

\begin{proof}[Proof of Theorem \ref{t solution Molnar}] 
According to the terminology in \cite{BenAliPeraltaRamirez}, every mapping $T$ in the hypothesis of this Theorem is a extreme-strong-local $^*$-automorphism on $B(H)$. Theorem 4.3 in \cite{BenAliPeraltaRamirez} implies that $T$ is a (continuous) Jordan $^*$-homomorphism on $B(H)$. A remarkable result of E. St{\o}rmer, assures that $T$ is a $^*$-homomorphism or a $^*$-anti-homomorphism (cf. \cite{Stor1965}, see also \cite[Theorem 10]{Kad51}, \cite{Bre89}, and \cite[Appendix]{Mol07}).\smallskip

We note that every extreme-strong-local $^*$-automorphism on an atomic von Neumann algebra is unital, and hence non-zero. Since $T$ is a Jordan $^*$-homomorphism, its kernel, $\ker(T),$ is a norm closed Jordan $^*$-ideal of $B(H).$ Theorem 5.3 in \cite{CivYood65} assures that $\ker(T)$ is a closed (associative) ideal of $B(H)$. Since $B(H)$ is a factor and $T$ is non-zero, we deduce that $\ker(T)=\{0\},$ and hence $T$ is a Jordan $^*$-monomorphism.\smallskip

Suppose now that $H$ is finite dimensional. The arguments above show that $T$ is a Jordan $^*$-automorphism, and hence a $^*$-automorphism or a $^*$-anti-automorphism. Furthermore, by \cite[Theorem A.8]{Mol07}, there exists a unitary $u_0\in B(H)$ satisfying $$T(a) = u_0 a u_0^*, \hbox{ or } T(a) = u_0  a^{t} u_0^*,$$ for every $a\in B(H)$. Suppose that $T(a) = u_0  a^{t} u_0^*,$ for every $a\in B(H)$. Theorem \ref{t non bilocally unitarily equivalent operators} implies the existence of an operator $a_0\in B(H)$ such that $a_0^{t}$  is not bilocally unitarily equivalent to $a_0$, that is, there exist $\xi_0$ and $\eta_0$ in $H$ satisfying \begin{equation}\label{eq non zero at xi0 eta0} \| a_0^t (\eta_0) - u a_0 u^* (\eta_0)\| + \| a_0^t (\xi_0) -u a_0 u^* (\xi_0)\| \neq 0,
 \end{equation} for every unitary $u\in B(H).$ Let us define $\widetilde{\xi}_0 = u_0 (\xi_0)$ and $\widetilde{\eta}_0 = u_0 (\eta_0)$ in $H$. It follows from \eqref{eq non zero at xi0 eta0} that $$\| T(a_0) (\widetilde{\eta}_0)  - u_0 u a_0 u^* u_0^* (\widetilde{\eta}_0)\| + \| T(a_0) (\widetilde{\xi}_0) - u_0 u a_0 u^* u_0^* (\widetilde{\xi}_0)\|$$
$$=\| u_0 a_0^t u_0^* (\widetilde{\eta}_0)  - u_0 u a_0 u^* u_0^* (\widetilde{\eta}_0)\| + \| u_0 a_0^t u_0^* (\widetilde{\xi}_0) - u_0 u a_0 u^* u_0^* (\widetilde{\xi}_0)\|$$ $$ = \| u_0 a_0^t  ({\eta}_0)  - u_0 u a_0 u^*  ({\eta}_0)\| + \| u_0 a_0^t  ({\xi}_0) - u_0 u a_0 u^*  ({\xi}_0)\| $$ $$= \| a_0^t  ({\eta}_0)  - u a_0 u^*  ({\eta}_0)\| + \| a_0^t  ({\xi}_0) - u a_0 u^*  ({\xi}_0)\| \neq 0,$$ for every unitary $u\in B(H),$ which contradicts that $T$ satisfies the  3-local property at the point $a_0$. Therefore, $T$ is a $^*$-automorphism.\smallskip

We finally assume that $H$ is infinite dimensional (non-necessarily separable) and $T$ is a Jordan $^*$-monomorphism satisfying the 3-local property. We have already commented, in the first paragraph, that $T$ is either a $^*$-homomorphism or a $^*$-anti-homomorphism. Arguing by contradiction, we suppose that $T$ is a $^*$-anti-homomorphism.\smallskip

Let $p_1,$ $\ldots$, $p_n$ be mutually orthogonal minimal projections in $B(H)$ with $n\geq 3$. Clearly, $T(p_1), \ldots$ $, T(p_n)$ are mutually orthogonal minimal projections in $B(H)$ (compare Lemma \ref{l tech lemma rank one}). Let $\displaystyle p = \sum_j p_j$. Since $T$ is a $^*$-anti-monomorphism, it maps $pB(H)p\equiv M_n (\mathbb{C})$ into $T(p) B(H) T(p)\equiv M_n (\mathbb{C})$ and $T|_{_{pB(H)p}} : pB(H)p\to T(p) B(H) T(p) $ is a $^*$-anti-monomorphism. In this case, $p_j = \zeta_j\otimes \zeta_j$ and $T(p_j) = \kappa_j\otimes \kappa_j$, where $\{\zeta_1,\ldots, \zeta_n\}$ and $\{\kappa_1,\ldots, \kappa_n\}$ are orthonormal systems in $H$. Let us take a unitary $u_0\in B(H)$ mapping each $\zeta_j$ to $\kappa_j.$\smallskip

The mapping $$u_0^* T|_{_{pB(H)p}} u_0 : pB(H)p \to pB(H)p\equiv M_n (\mathbb{C})$$ also is a $^*$-anti-monomorphism. We deduce from \cite[Theorem A.8]{Mol07}, that there exists a unitary $\widetilde{v}_0: p (H) \to p(H)$ such that $$ u_0^* T(x) u_0 = \widetilde{v}_0^* x^t \widetilde{v}_0,
$$ for every $x\in pB(H)p \equiv B(p(H)).$ We can find a unitary $v_0\in B(H)$ satisfying $v_0 |_{p(H)} =  \widetilde{v}_0,$ and hence \begin{equation}\label{eq unitary in finite dimensional 1} u_0^* T(x) u_0 = {v}_0^* x^t {v}_0, \hbox{ or equivalently, }  T(x) = u_0 {v}_0^* x^t {v}_0 u_0^*,
\end{equation} for every $x\in pB(H)p \equiv B(p(H)).$ Applying the second statement in Theorem \ref{t non bilocally unitarily equivalent operators}, we find an element $a_0\in B(H)$ with $a_0= pa_0 p,$ which is not bilocally unitarily equivalent to its transpose in $B(H)$, that is, there exist $\xi_0$ and $\eta_0$ in $H$ satisfying $$\| a_0^t (\eta_0) - u a_0 u^* (\eta_0)\| + \| a_0^t (\xi_0) -u a_0 u^* (\xi_0)\| \neq 0,$$ for every unitary $u\in B(H),$ a result that, by \eqref{eq unitary in finite dimensional 1}, contradicts the 3-local property of $T$, because $a_0 \in pB(H) p.$
\end{proof}

It should be remarked here that Theorem \ref{t solution Molnar} also completes the conclusions in \cite[\S 4]{BenAliPeraltaRamirez}.


\begin{thebibliography}{22}



\bibitem{BenAliPeraltaRamirez} A. Ben Ali Essaleh, A.M. Peralta, M.I. Ram{\'\i}rez, Weak-local derivations and homomorphisms on C$^*$-algebras, preprint 2014. arxiv:1411.4795v2.



\bibitem{Bre89} M. Bresar, Jordan Mappings of Semiprime Rings, \emph{J. Algebra}, \textbf{127}, 1, 218-228 (1989).


\bibitem{BreSemrl95} M. Bre\v{s}ar, P. \v{S}emrl, On local automorphisms and mappings that preserve idempotents, \emph{Studia Math.} \textbf{113}, no. 2, 101-108  (1995).


\bibitem{CivYood65} P. Civin, B. Yood, Lie and Jordan structures in Banach algebras, \emph{Pacific J. Math.}
\textbf{15}, 775-797 (1965).


\bibitem{GarTener2012} S.R. Garcia, J.E. Tener, Unitary equivalence of a matrix to its transpose, \emph{J. Operator Theory} \textbf{68}, no. 1, 179-203 (2012).


\bibitem{Hal54} P.R. Halmos, \emph{A Linear Algebra Problem Book}, Dolciani Math. Exp., vol. 16, Math. Assoc. America, Washington, DC 1995.

\bibitem{IkraAbdi2012} Kh. D. Ikramov, A.K. Abdikalykov, On unitary transposable matrices of order three, \emph{Translation of Mat. Zametki} \textbf{91}, no. 4, 563-570  (2012). \emph{Math. Notes} \textbf{91}, no. 3-4, 528-534 (2012).



\bibitem{Kad51} R. V. Kadison, Isometries of Operator Algebras, \emph{Ann. of Math.}, Vol. \textbf{54}, No. 2, 325-338  (1951).






\bibitem{LarSou} D.R. Larson and A.R. Sourour, Local derivations and local automorphisms of $B(X)$, \emph{Proc. Sympos. Pure Math.} \textbf{51}, Part 2, Providence, Rhode Island 1990, pp. 187-194.




\bibitem{Mol07} L. Moln\'{a}r, \emph{Selected Preserver Problems on Algebraic Structures of Linear Operators and on Function Spaces, Lecture Notes in Mathematics 1895}, Springer-Verlag, Berlin Heidelberg, 2007.

\bibitem{Mol2014} L. Moln\'{a}r, Bilocal $^*$-automorphisms of $B(H),$ \emph{Arch. Math.} \textbf{102}, 83-89 (2014).


\bibitem{Pe2014} A.M. Peralta, A note on 2-local representations of C$^*$-algebras, to appear in \emph{Operators and Matrices}.






\bibitem{Specht} W. Specht, Zur Theorie der Matrizen. II., \emph{Jber. Deutsch. Math. Verein.} \textbf{50}, 19-23 (1940).



\bibitem{Stor1965} E. St{\o}rmer, On the Jordan Structure of C*-Algebras, Trans. Amer. Math. Soc., \textbf{120}, No. 3, 438-447 (1965).

\bibitem{ZhuXiong97} C. Xiong, J. Zhu, Bilocal derivations of standard operator algebras, \emph{Proc. Amer. Math. Soc.} \textbf{125}, 1367-1370 (1997).

\end{thebibliography}
\end{document}